\documentclass[a4paper]{article}
\usepackage{hyphenat}
\usepackage{amsmath}
\usepackage{amsfonts}
\usepackage{amssymb}
\usepackage[a4paper,left=2.5cm,right=2.5cm,top=2.5cm, bottom=3cm]{geometry}
\usepackage{float}
\usepackage{fourier}
\usepackage{enumerate}
\usepackage{tikz}
\usepackage{xcolor}
\usepackage{wasysym}

\newtheorem{lemma}{Lemma}[section]
\newtheorem{theorem}[lemma]{Theorem}
\newtheorem{corollary}[lemma]{Corollary}

\newtheorem{conjecture}[lemma]{Conjecture}

\newenvironment{proofof}{\noindent}{\hfill$\Box$\medskip}
\newenvironment{proof}{\noindent{\it Proof: }}{\hfill\hfill$\Box$\medskip}

\newcommand{\cont}{\subseteq}
\renewcommand{\u}{\cup}\renewcommand{\i}{\cap}

\newcommand{\emp}{\emptyset}

\renewcommand{\mod}{\;{\rm mod}\,}

\newcommand{\lf}{\left\lfloor}\newcommand{\rf}{\right\rfloor}
\newcommand{\lc}{\left\lceil}\newcommand{\rc}{\right\rceil}

\newcommand{\defin}{\textbf}

\newcommand{\fq}{\mathbb{F}_q}
\newcommand{\fqt}{\fq^3}

\newcommand{\pq}{PG(2,q)}
\newcommand{\C}{\mathcal{C}}
\renewcommand{\L}{\mathcal{L}}
\renewcommand{\l}[1]{L(#1)}
\newcommand{\sq}{{\rm SQ}}

\renewcommand{\l}[1]{\overleftrightarrow{#1}}
\newcommand{\zn}{\mathbb{Z}_n}

\newcommand{\xid}{\xi_D}
\newcommand{\znD}{\mathbb{D}_n}
\newcommand{\diag}{\mathbb{D}}
\newcommand{\Z}{\mathbb{Z}}

\title{Toroidal boards and code covering}
\author{\noindent Jo\~ao Paulo Costalonga\footnote{The author was partially supported by CNPq, grant 478053/2013-4} \\
Universidade Federal do Esp\'irito Santo\\Av. Fernando Ferrari, 514,\\ UFES - Campus de Goiabeiras, CCE - Depto. de Matem\'atica\\
29075-910 - Vit\'oria - ES - Brazil\\
joaocostalonga@gmail.com
}

\begin{document}

\maketitle

\begin{abstract}
We denote by $\mathbb{F}_q$ the field with $q$ elements. A radius-$r$ extended ball with center in a $1$-dimensional vector subspace $V$ of $\mathbb{F}_q^3$ is the set of elements of $\mathbb{F}_q^3$ with Hamming distance to $V$ at most $r$. We define $c(q)$ as the size of a minimum covering of $\fqt$ by radius-$1$ extended balls. We define a semiqueen as a piece of a toroidal chessboard that extends the covering range of a rook by the southwest-northeast diagonal containing it. Let $\xi_D(n)$ be the minimum number of semiqueens of the $n\times n$ toroidal board necessary to cover the entire board except possibly for the southwest-northeast diagonal. We prove that, for $q\ge 7$, $c(q)=\xi_D(q-1)+2$. Moreover, our proof exhibits a method to build such covers of $\mathbb{F}_q^3$ from the semiqueen coverings of the board. With this new method, we determine $c(q)$ for the odd values of $q$ and improve both existing bounds for the even case.
\end{abstract}

\noindent Key words: Code covering, short covering, toroidal board, projective geometry, semiqueen. MSC 2010: 05B40 and 94B65.

\section{Introduction}
The problem of finding minimum coverings of $\fq^n$ with radius-$r$ balls in the Hamming distance is classic in code theory. There is a book on the subject~\cite{Cohen} and an updated table with the known bounds for the sizes of such coverings~\cite{Keri}. In \cite{Geronimo}, a variation of this problem was introduced: a radius-$r$ \defin{extended ball} with center in a $1$-dimensional vector subspace $V$ of $\fq^n$ is the set of elements of $\fq^n$ with Hamming distance to $V$ at most $r$. We define $c_q(n,r)$ as the size of a minimum covering of $\fq^n$ by radius-$r$ extended balls; such a minimum covering is called a \defin{short covering}. 

In \cite{Mendes} some interesting reasons to study short coverings are listed. One is that short coverings were used to construct record breaking classical coverings (in \cite{Mendes} it is proved using short covering that the minimum number of radius-$7$ hamming balls necessary to cover $\mathcal F_5^{10}$ is $9$). Another is that they respond well to some heuristic methods and give an economical way in terms of memory to store codes. A third one is that they seem to have more interesting mathematical properties than the classic coverings, like more compatibility with the algebraic structure of the vector space $\fq^n$ and connection with other structures (see \cite{Martinhao,Neto,Nakaoka} for examples).

Few values of $c_q(n,r)$ are known. Here, our concern is the values of $c(q):=c_q(3,1)$. Some work \cite{Martinhao, Nakaoka, Neto} proved bounds for $c(q)$. In Corollary \ref{cvalues}, we establish $c(q)$ for odd values of $q$ and improve both existing bounds for the even values. In order to do this, we introduce a relation between short coverings, projective spaces and toroidal (chess)boards. The number of rooks needed to cover an $n\times n$ toroidal board is well known, clearly $n$. Some studies on covering and packing of queens in the toroidal boards were made by \cite{Burger2001} and \cite{Burger2003}. We introduce a piece with range between a rook and a queen, as described next.

We will use $1,\dots,n$ as standard representatives for the classes of $\zn$. The toroidal $n\times n$ board will be modeled by $\zn^2$, with the first coordinate indexing the column and the second the row, in such a way that $(1,1)$ corresponds to the southwestern square and $(n,n)$ to the northeastern square (the orientation is similar to a Cartesian plane). The \defin{diagonal} of $(a,b)\in \zn^2$ is the set $D(a,b):=\{(a+t,b+t):t\in\zn\}$. The \defin{vertical} and \defin{horizontal} lines of $(a,b)\in \zn^2$ are respectively defined by $V(a,b):=\{(a,t):t\in \zn\}$ and $H(a,b):=\{(t,b):t\in \zn\}$. The \defin{semiqueen} of $(a,b)\in \zn^2$ is the set $\sq(a,b):=D(a,b)\u V(a,b)\u H(a,b)$. 

We denote by $\znD$ the $n\times n$ board without the southwest-northeast diagonal: $\znD:=\zn^2-D(1,1)$. We also denote by $\xi(n)$ and $\xid(n)$ the respective sizes of minimum coverings of $\zn^2$ and $\znD$ by semiqueens of $\zn^2$.  Next we state our main results. The next theorem establishes a relation between the values of $c(q)$ and $\xid(q-1)$.

\begin{theorem}\label{main}
For a prime power $q\ge 5$, $c(q)=\xid(q-1)+2$. Moreover, for $q\ge 7$, there is an algorithm for building minimum coverings of $\fqt$ by radius-$1$ extended balls from coverings of $\diag_{q-1}$ by semiqueens and vice-versa.
\end{theorem}

The proof of the first part of Theorem \ref{main} is a construction that gives the algorithm for the second part. The proof for Theorem \ref{main} is concluded in Section \ref{sec-projective}. There are certain difficulties in using this technique for higher dimensions than $3$. One is to make a more general version of Lemma \ref{particular cover} and another is to study the coverings of higher-dimensional boards. The next theorem establishes values and bounds for $\xi(n)$:
	
\begin{theorem}\label{xivalues}
Let $n$ be a positive integer.
\begin{enumerate}
\item [(a)] If $n\equiv 2\mod 4$, then $\xi(n)=n/2$.
\item [(b)] If $n\equiv 0\mod 4$, then $\xi(n)=1+n/2$.
\item [(c)] If $n$ is odd, then $\frac{n+1}{2}\le \xi(n)\le \frac{2n+1}{3}$.
\end{enumerate}
\end{theorem}

For values of $\xid(n)$, we have:

\begin{theorem}\label{xidvalues}
Let $n$ be a positive integer.
\begin{enumerate}
\item [(a)] If $n$ is even, then $\xid(n)=n/2$.
\item [(b)] If $n$ is odd, then $\frac{n+1}{2} \le \xid(n)\le \frac{2n+1}{3}$.
\end{enumerate}
\end{theorem}

From Theorems \ref{main} and \ref{xidvalues} and the known values of $c(3)$ and $c(4)$ of \cite{Geronimo}(see also Section \ref{sec-ilp}), we have:

\begin{corollary}\label{cvalues}
Let $q\ge 3$ be a prime power.
\begin{enumerate}
\item [(a)] If $q$ is odd, then $c(q)=\frac{q+3}{2}$.
\item [(b)] If $q$ is even, then $\frac{q+4}{2}\le c(q) \le \frac{2q+5}{3}$.
\end{enumerate}
\end{corollary}

The upper bound $c(q)\le\frac{q+3}{2}$ in Corollary \ref{cvalues} was proved by Martinh\~ao and Carmelo in \cite{Martinhao} for $q\equiv 3 \mod 4$. The same result for $q\equiv 1 \mod 4$ was recently proved independently of this work by Martinh\~ao~\cite{Mart-tese}. The upper bound in item (b) of Corollary \ref{cvalues} improves the previous one 
$  c(q)\le 6{\scriptstyle \lc \frac{q-1}{9} \rc} + 6 {\scriptscriptstyle \lc \log_4\left( \frac{q-1}{3}\right) \rc} +3$, set in \cite{Neto}. The lower bounds of Corollary \ref{cvalues} improve the bound $c(q)\ge (q+1)/2$ set in \cite{Nakaoka}. The next theorem gives us better upper bounds for some even values of $q$:
\begin{theorem}\label{prodvalues}
For positive odd integers $m$ and $n$:
\begin{enumerate}
\item [(a)] $\xi(mn)\le m\xi(n)$. 
\item [(b)] If $q$ is a power of $2$ and $q=mn+1\ge 7$, then $c(q)\le m\xi(n)+2$.
\end{enumerate}
\end{theorem}

Theorems \ref{xivalues}, \ref{xidvalues} and \ref{prodvalues} are proved in Section \ref{sec-proofs}. In Section \ref{sec-ilp}, we use an integer linear programming (ILP) formulation to compute $\xi(n)$, $\xid(n)$ and $c(q)$ for small values of $q$ and $n$ not covered by our results. There are still few known values for $\xi(n)$ with $n$ odd. Next, we state some conjectures:
	
\begin{conjecture}
If $p$ is a prime number, then $\xi(p)=\lf \frac{2p+1}{3}\rf$.
\end{conjecture}

\begin{conjecture}
If $n$ is an odd positive integer, then 

\centerline{$\xi(n)=\min\{(n/m)\xi(m):$ $m$ divides $n\}$.}
\end{conjecture}

\begin{conjecture}
If $n$ is an odd positive integer, then 

\centerline{$\xi(n)=\min\{(n/p)\xi(p):$ $p$ is a prime divisor of $n\}$.}
\end{conjecture}

\begin{conjecture}For $n$ assuming positive integer values, 
${\displaystyle \lim_{n\rightarrow\infty}}\frac{\xi(2n+1)}{2n+1}=\frac{2}{3}$.
\end{conjecture}

\section{Proofs of Theorems \ref{xivalues}, \ref{xidvalues} and \ref{prodvalues}}\label{sec-proofs}
In this section, we prove Theorems \ref{xivalues}, \ref{xidvalues} and \ref{prodvalues}. We will prove some lemmas and establish some concepts first.

Next, we extend our definitions for more general groups than $\zn$. Let $G$ be a finite abelian group and $(a,b)\in G^2$. We define the respective \defin{diagonal}, \defin{vertical} and \defin{horizontal} lines and \defin{semiqueen} of $(a,b)\in G^2$ as follows:
\begin{itemize}
 \item $D(a,b):=D_G(a,b):=\{(ta,tb):t\in G\}$;
 \item $H(a,b):=H_G(a,b):=\{(t,b):t\in G\}$;
 \item $V(a,b):=V_G(a,b):=\{(a,t):t\in G\}$;
 \item $\sq(a,b):=\sq_G(a,b):=D(a,b)\u H(a,b)\u V(a,b)$.
\end{itemize}
For $X\cont G^2$, we define $\sq(X)$ as the union of all semiqueens of the form $\sq(x)$ with $x\in X$. In an analogous way we define $D(X)$, $H(X)$ and $V(X)$.

We define $\diag(G):=G^2-D(1_G,1_G)$ and denote by $\xi(G)$ and $\xid(G)$ the respective sizes of a minimum covering $G^2$ and $\diag(G)$ by semiqueens of $G^2$. Suppose that $\varphi:G\rightarrow H$ is a group isomorphism and define $\Phi(a,b)=(\varphi(a),\varphi(b))$ for $(a,b)\in G^2$. It is clear that for $x\in G^2$, $\sq_H(\Phi(x))=\Phi(\sq_G(x))$. Therefore:

\begin{lemma}\label{isomorphism}
If $G$ and $H$ are isomorphic finite abelian groups, then $\xi(G)=\xi(H)$ and $\xid(G)=\xid(H)$.
\end{lemma}

\begin{lemma}\label{lower bound}
For each finite abelian group $G$, $\xid(G)\ge(|G|-1)/2$ and $\xi(G)\ge |G|/2$.
\end{lemma}
\begin{proof}
Write $n:=|G|$. Let $\{\sq(x_1),\dots,\sq(x_k)\}$ be a minimum covering of $\diag(G)$ by semiqueens. It is clear that $k=\xid(G)\le\xi(G)<n$. So, we may choose a vertical line $L$ of $G^2$ avoiding $V(x_1),\dots, V(x_k)$. Note that $C:=(L\i \diag(G))-(H(x_1)\u\cdots\u H(x_k))$ has at least $n-k-1$ elements, which must be covered by $D(x_1),\dots,D(x_k)$. Since each diagonal intersects $C$ in one element, $k\ge n-k-1$ and $\xid(G)=k\ge (n-1)/2$. Analogously, we can prove that $\xi(G)\ge n/2$.
\end{proof}

For $k\in \mathbb{Z}_+$, we define a function $\delta_k: \Z_k^2\rightarrow\Z_k$ by $\delta_k(a,b)=b-a$ for each $(a,b)\in \Z_k^2$. We will use this funtion in the proofs that follow in this section. The proof of the next lemma is elementary.

\begin{lemma}\label{delta function}Le $k\in \mathbb Z_+$. If $(a,b),(c,d)\in \Z_k^2$, then $(c,d)\in D(a,b)$ if and only if $\delta_k(c,d)=\delta_k(a,b)$.
\end{lemma}

\begin{lemma}\label{xi2mod4}
If $n$ is a positive integer and $n\equiv 2 \mod 4$, then $\xid(n)=\xi(n)=n/2$.
\end{lemma}
\begin{proof}By Lemma \ref{lower bound}, it is enough to find a covering of $\zn^2$ with $n/2$ semiqueens. Define $X:=\{(2,n),(4,n-2),\dots,(n,2)\}$. Let us check that $\{\sq(x):x\in X\}$ covers the board. This covering is illustrated for $n=6$ in Figure \ref{fig-even-values}. Note that $\delta_n(X)=\{n-2,n-6,\dots,2-n\}$. Since $n\equiv 2\mod 4$, it follows that $\delta_n(X)$ is the set of the even elements of $\zn$.

Now let $(a,b)\in \zn^2$. If both $b$ and $a$ are odd, then $\delta_n(a,b)$ is even and, therefore, $\delta_n(a,b)\in \delta_n(X)$ and $(a,b)\in D(X)\cont \sq(X)$. Otherwise, if one of $a$ or $b$ is even, it is clear that $(a,b)$ is in the vertical or horizontal line of an element of $X$. Therefore, $\{\sq(x):x\in X\}$ covers the board and the lemma is true.
\end{proof}

\begin{figure}[h]\centering
\begin{tikzpicture}[scale=0.5]
	\draw (0.5cm,0.5cm)--(0.5cm,6.5cm); \draw (1.5cm,0.5cm)--(1.5cm,6.5cm);\draw (2.5cm,0.5cm)--(2.5cm,6.5cm);
	\draw (3.5cm,0.5cm)--(3.5cm,6.5cm); \draw (4.5cm,0.5cm)--(4.5cm,6.5cm);
	\draw (5.5cm,0.5cm)--(5.5cm,6.5cm); \draw (6.5cm,0.5cm)--(6.5cm,6.5cm); 

	\draw (0.5cm,0.5cm)--(6.5cm,0.5cm); \draw (0.5cm,1.5cm)--(6.5cm,1.5cm);\draw (0.5cm,2.5cm)--(6.5cm,2.5cm);
	\draw (0.5cm,3.5cm)--(6.5cm,3.5cm); \draw (0.5cm,4.5cm)--(6.5cm,4.5cm);
	\draw (0.5cm,5.5cm)--(6.5cm,5.5cm); \draw (0.5cm,6.5cm)--(6.5cm,6.5cm); 

	{\color{orange}
		\node at(2cm,6cm){$\blacklozenge$};
		\draw[thick](2cm,0.5cm)--(2cm,6.5cm);
		\draw[thick](0.5cm,6cm)--(6.5cm,6cm);
		\draw[thick](0.5cm,4.5cm)--(2.5cm,6.5cm);
		\draw[thick](2.5cm,0.5cm)--(6.5cm,4.5cm);
	}
	{\color{red}
		\node at(4cm,4cm){$\blacksquare$};
		\draw[thick](4cm,0.5cm)--(4cm,6.5cm);
		\draw[thick](0.5cm,4cm)--(6.5cm,4cm);
		\draw[thick](0.5cm,0.5cm)--(6.5cm,6.5cm);
	}
	{\color{blue}
		\node at(6cm,2cm){$\CIRCLE$};
		\draw[thick](6cm,0.5cm)--(6cm,6.5cm);
		\draw[thick](0.5cm,2cm)--(6.5cm,2cm);
		\draw[thick](4.5cm,0.5cm)--(6.5cm,2.5cm);
		\draw[thick](0.5cm,2.5cm)--(4.5cm,6.5cm);
	}
	\fill (01cm,06cm) circle (0.10cm);\fill (02cm,06cm) circle (0.10cm);\fill (03cm,06cm) circle (0.10cm);
	\fill (01cm,05cm) circle (0.10cm);\fill (02cm,05cm) circle (0.10cm);\fill (03cm,05cm) circle (0.10cm);
	\fill (01cm,04cm) circle (0.10cm);\fill (02cm,04cm) circle (0.10cm);\fill (03cm,04cm) circle (0.10cm);
	\fill (01cm,03cm) circle (0.10cm);\fill (02cm,03cm) circle (0.10cm);\fill (03cm,03cm) circle (0.10cm);
	\fill (01cm,02cm) circle (0.10cm);\fill (02cm,02cm) circle (0.10cm);\fill (03cm,02cm) circle (0.10cm);
	\fill (01cm,01cm) circle (0.10cm);\fill (02cm,01cm) circle (0.10cm);\fill (03cm,01cm) circle (0.10cm);

	\fill (04cm,06cm) circle (0.10cm);\fill (05cm,06cm) circle (0.10cm);\fill (06cm,06cm) circle (0.10cm);
	\fill (04cm,05cm) circle (0.10cm);\fill (05cm,05cm) circle (0.10cm);\fill (06cm,05cm) circle (0.10cm);
	\fill (04cm,04cm) circle (0.10cm);\fill (05cm,04cm) circle (0.10cm);\fill (06cm,04cm) circle (0.10cm);
	\fill (04cm,03cm) circle (0.10cm);\fill (05cm,03cm) circle (0.10cm);\fill (06cm,03cm) circle (0.10cm);
	\fill (04cm,02cm) circle (0.10cm);\fill (05cm,02cm) circle (0.10cm);\fill (06cm,02cm) circle (0.10cm);
	\fill (04cm,01cm) circle (0.10cm);\fill (05cm,01cm) circle (0.10cm);\fill (06cm,01cm) circle (0.10cm);

	{\color{orange}
		\node at(2cm,6cm){$\blacklozenge$};
	}
	{\color{red}
		\node at(4cm,4cm){$\blacksquare$};
	}
	{\color{blue}
		\node at(6cm,2cm){$\CIRCLE$};
	}
\end{tikzpicture}$\qquad$
\begin{tikzpicture}[scale=0.4]
	\fill [gray!30](00.5cm,00.5cm)--(01.5cm,00.5cm)--(01.5cm,01.5cm)--(00.5cm,01.5cm)--cycle;
	\fill [gray!30](01.5cm,01.5cm)--(02.5cm,01.5cm)--(02.5cm,02.5cm)--(01.5cm,02.5cm)--cycle;
	\fill [gray!30](02.5cm,02.5cm)--(03.5cm,02.5cm)--(03.5cm,03.5cm)--(02.5cm,03.5cm)--cycle;
	\fill [gray!30](03.5cm,03.5cm)--(04.5cm,03.5cm)--(04.5cm,04.5cm)--(03.5cm,04.5cm)--cycle;
	\fill [gray!30](04.5cm,04.5cm)--(05.5cm,04.5cm)--(05.5cm,05.5cm)--(04.5cm,05.5cm)--cycle;
	\fill [gray!30](05.5cm,05.5cm)--(06.5cm,05.5cm)--(06.5cm,06.5cm)--(05.5cm,06.5cm)--cycle;
	\fill [gray!30](06.5cm,06.5cm)--(07.5cm,06.5cm)--(07.5cm,07.5cm)--(06.5cm,07.5cm)--cycle;
	\fill [gray!30](07.5cm,07.5cm)--(08.5cm,07.5cm)--(08.5cm,08.5cm)--(07.5cm,08.5cm)--cycle;
	\fill [gray!30](08.5cm,08.5cm)--(09.5cm,08.5cm)--(09.5cm,09.5cm)--(08.5cm,09.5cm)--cycle;
	\fill [gray!30](09.5cm,09.5cm)--(10.5cm,09.5cm)--(10.5cm,10.5cm)--(09.5cm,10.5cm)--cycle;
	\fill [gray!30](10.5cm,10.5cm)--(11.5cm,10.5cm)--(11.5cm,11.5cm)--(10.5cm,11.5cm)--cycle;
	\fill [gray!30](11.5cm,11.5cm)--(12.5cm,11.5cm)--(12.5cm,12.5cm)--(11.5cm,12.5cm)--cycle;

	\draw (0.5cm,0.5cm)--(0.5cm,12.5cm); \draw (1.5cm,0.5cm)--(1.5cm,12.5cm); \draw (2.5cm,0.5cm)--(2.5cm,12.5cm);
	\draw (3.5cm,0.5cm)--(3.5cm,12.5cm); \draw (4.5cm,0.5cm)--(4.5cm,12.5cm); \draw (5.5cm,0.5cm)--(5.5cm,12.5cm);
	\draw (6.5cm,0.5cm)--(6.5cm,12.5cm); \draw (7.5cm,0.5cm)--(7.5cm,12.5cm); \draw (8.5cm,0.5cm)--(8.5cm,12.5cm);
	\draw (9.5cm,0.5cm)--(9.5cm,12.5cm); \draw (10.5cm,0.5cm)--(10.5cm,12.5cm); \draw (11.5cm,0.5cm)--(11.5cm,12.5cm);
	\draw (12.5cm,0.5cm)--(12.5cm,12.5cm);
	\draw (0.5cm,0.5cm)--(12.5cm,0.5cm); \draw (0.5cm,1.5cm)--(12.5cm,1.5cm); \draw (0.5cm,2.5cm)--(12.5cm,2.5cm);
	\draw (0.5cm,3.5cm)--(12.5cm,3.5cm); \draw (0.5cm,4.5cm)--(12.5cm,4.5cm); \draw (0.5cm,5.5cm)--(12.5cm,5.5cm);
	\draw (0.5cm,6.5cm)--(12.5cm,6.5cm); \draw (0.5cm,7.5cm)--(12.5cm,7.5cm); \draw (0.5cm,8.5cm)--(12.5cm,8.5cm);
	\draw (0.5cm,9.5cm)--(12.5cm,9.5cm); \draw (0.5cm,10.5cm)--(12.5cm,10.5cm); \draw (0.5cm,11.5cm)--(12.5cm,11.5cm);
	\draw (0.5cm,12.5cm)--(12.5cm,12.5cm);

	{\color{orange}
		\fill (2cm,12cm)circle(0.3cm);
		\draw[thick](2cm,0.5cm)--(2cm,12.5cm);
		\draw[thick](0.5cm,12cm)--(12.5cm,12cm);
		\draw[thick](0.5cm,10.5cm)--(2.5cm,12.5cm);
		\draw[thick](2.5cm,0.5cm)--(12.5cm,10.5cm);
	}
	{\color{red}
		\fill (4cm,10cm)circle(0.3cm);
		\draw[thick](4cm,0.5cm)--(4cm,12.5cm);
		\draw[thick](0.5cm,10cm)--(12.5cm,10cm);
		\draw[thick](0.5cm,6.5cm)--(6.5cm,12.5cm);
		\draw[thick](6.5cm,0.5cm)--(12.5cm,6.5cm);
	}
	{\color{blue}
		\fill (6cm,8cm)circle(0.3cm);
		\draw[thick](6cm,0.5cm)--(6cm,12.5cm);
		\draw[thick](0.5cm,8cm)--(12.5cm,8cm);
		\draw[thick](0.5cm,2.5cm)--(10.5cm,12.5cm);
		\draw[thick](10.5cm,0.5cm)--(12.5cm,2.5cm);
	}
	{\color{green}
		\fill (8cm,4cm)circle(0.3cm);
		\draw[thick](8cm,0.5cm)--(8cm,12.5cm);
		\draw[thick](0.5cm,4cm)--(12.5cm,4cm);
		\draw[thick](4.5cm,0.5cm)--(12.5cm,8.5cm);
		\draw[thick](0.5cm,8.5cm)--(4.5cm,12.5cm);
	}
	
	{\color{magenta}
		\fill (10cm,2cm)circle(0.3cm);
		\draw[thick](10cm,0.5cm)--(10cm,12.5cm);
		\draw[thick](0.5cm,2cm)--(12.5cm,2cm);
		\draw[thick](8.5cm,0.5cm)--(12.5cm,4.5cm);
		\draw[thick](0.5cm,4.5cm)--(8.5cm,12.5cm);
	}
	
	{\color{cyan}
		\fill (12cm,6cm)circle(0.3cm);
		\draw[thick](12cm,0.5cm)--(12cm,12.5cm);
		\draw[thick](0.5cm,6cm)--(12.5cm,6cm);
	}

	\fill (01cm,12cm) circle (0.10cm);\fill (02cm,12cm) circle (0.10cm);\fill (03cm,12cm) circle (0.10cm);
	\fill (01cm,11cm) circle (0.10cm);\fill (02cm,11cm) circle (0.10cm);\fill (03cm,11cm) circle (0.10cm);
	\fill (01cm,10cm) circle (0.10cm);\fill (02cm,10cm) circle (0.10cm);\fill (03cm,10cm) circle (0.10cm);
	\fill (01cm,09cm) circle (0.10cm);\fill (02cm,09cm) circle (0.10cm);\fill (03cm,09cm) circle (0.10cm);
	\fill (01cm,08cm) circle (0.10cm);\fill (02cm,08cm) circle (0.10cm);\fill (03cm,08cm) circle (0.10cm);
	\fill (01cm,07cm) circle (0.10cm);\fill (02cm,07cm) circle (0.10cm);\fill (03cm,07cm) circle (0.10cm);
	\fill (01cm,06cm) circle (0.10cm);\fill (02cm,06cm) circle (0.10cm);\fill (03cm,06cm) circle (0.10cm);
	\fill (01cm,05cm) circle (0.10cm);\fill (02cm,05cm) circle (0.10cm);\fill (03cm,05cm) circle (0.10cm);
	\fill (01cm,04cm) circle (0.10cm);\fill (02cm,04cm) circle (0.10cm);\fill (03cm,04cm) circle (0.10cm);
	\fill (01cm,03cm) circle (0.10cm);\fill (02cm,03cm) circle (0.10cm);                                  
	\fill (01cm,02cm) circle (0.10cm);                                  \fill (03cm,02cm) circle (0.10cm);
	                                  \fill (02cm,01cm) circle (0.10cm);\fill (03cm,01cm) circle (0.10cm);
	                                 
	\fill (04cm,12cm) circle (0.10cm);\fill (05cm,12cm) circle (0.10cm);\fill (06cm,12cm) circle (0.10cm);
	\fill (04cm,11cm) circle (0.10cm);\fill (05cm,11cm) circle (0.10cm);\fill (06cm,11cm) circle (0.10cm);
	\fill (04cm,10cm) circle (0.10cm);\fill (05cm,10cm) circle (0.10cm);\fill (06cm,10cm) circle (0.10cm);
	\fill (04cm,09cm) circle (0.10cm);\fill (05cm,09cm) circle (0.10cm);\fill (06cm,09cm) circle (0.10cm);
	\fill (04cm,08cm) circle (0.10cm);\fill (05cm,08cm) circle (0.10cm);\fill (06cm,08cm) circle (0.10cm);
	\fill (04cm,07cm) circle (0.10cm);\fill (05cm,07cm) circle (0.10cm);\fill (06cm,07cm) circle (0.10cm);
	\fill (04cm,06cm) circle (0.10cm);\fill (05cm,06cm) circle (0.10cm);                                  
	\fill (04cm,05cm) circle (0.10cm);                                  \fill (06cm,05cm) circle (0.10cm);
	                                  \fill (05cm,04cm) circle (0.10cm);\fill (06cm,04cm) circle (0.10cm);
	\fill (04cm,03cm) circle (0.10cm);\fill (05cm,03cm) circle (0.10cm);\fill (06cm,03cm) circle (0.10cm);
	\fill (04cm,02cm) circle (0.10cm);\fill (05cm,02cm) circle (0.10cm);\fill (06cm,02cm) circle (0.10cm);
	\fill (04cm,01cm) circle (0.10cm);\fill (05cm,01cm) circle (0.10cm);\fill (06cm,01cm) circle (0.10cm);

	\fill (07cm,12cm) circle (0.10cm);\fill (08cm,12cm) circle (0.10cm);\fill (09cm,12cm) circle (0.10cm);
	\fill (07cm,11cm) circle (0.10cm);\fill (08cm,11cm) circle (0.10cm);\fill (09cm,11cm) circle (0.10cm);
	\fill (07cm,10cm) circle (0.10cm);\fill (08cm,10cm) circle (0.10cm);\fill (09cm,10cm) circle (0.10cm);
	\fill (07cm,09cm) circle (0.10cm);\fill (08cm,09cm) circle (0.10cm);                                  
	\fill (07cm,08cm) circle (0.10cm);                                  \fill (09cm,08cm) circle (0.10cm);
	                                  \fill (08cm,07cm) circle (0.10cm);\fill (09cm,07cm) circle (0.10cm);
	\fill (07cm,06cm) circle (0.10cm);\fill (08cm,06cm) circle (0.10cm);\fill (09cm,06cm) circle (0.10cm);
	\fill (07cm,05cm) circle (0.10cm);\fill (08cm,05cm) circle (0.10cm);\fill (09cm,05cm) circle (0.10cm);
	\fill (07cm,04cm) circle (0.10cm);\fill (08cm,04cm) circle (0.10cm);\fill (09cm,04cm) circle (0.10cm);
	\fill (07cm,03cm) circle (0.10cm);\fill (08cm,03cm) circle (0.10cm);\fill (09cm,03cm) circle (0.10cm);
	\fill (07cm,02cm) circle (0.10cm);\fill (08cm,02cm) circle (0.10cm);\fill (09cm,02cm) circle (0.10cm);
	\fill (07cm,01cm) circle (0.10cm);\fill (08cm,01cm) circle (0.10cm);\fill (09cm,01cm) circle (0.10cm);
	
	\fill (10cm,12cm) circle (0.10cm);\fill (11cm,12cm) circle (0.10cm);                                  
	\fill (10cm,11cm) circle (0.10cm);                                  \fill (12cm,11cm) circle (0.10cm);
	                                  \fill (11cm,10cm) circle (0.10cm);\fill (12cm,10cm) circle (0.10cm);
	\fill (10cm,09cm) circle (0.10cm);\fill (11cm,09cm) circle (0.10cm);\fill (12cm,09cm) circle (0.10cm);
	\fill (10cm,08cm) circle (0.10cm);\fill (11cm,08cm) circle (0.10cm);\fill (12cm,08cm) circle (0.10cm);
	\fill (10cm,07cm) circle (0.10cm);\fill (11cm,07cm) circle (0.10cm);\fill (12cm,07cm) circle (0.10cm);
	\fill (10cm,06cm) circle (0.10cm);\fill (11cm,06cm) circle (0.10cm);\fill (12cm,06cm) circle (0.10cm);
	\fill (10cm,05cm) circle (0.10cm);\fill (11cm,05cm) circle (0.10cm);\fill (12cm,05cm) circle (0.10cm);
	\fill (10cm,04cm) circle (0.10cm);\fill (11cm,04cm) circle (0.10cm);\fill (12cm,04cm) circle (0.10cm);
	\fill (10cm,03cm) circle (0.10cm);\fill (11cm,03cm) circle (0.10cm);\fill (12cm,03cm) circle (0.10cm);
	\fill (10cm,02cm) circle (0.10cm);\fill (11cm,02cm) circle (0.10cm);\fill (12cm,02cm) circle (0.10cm);
	\fill (10cm,01cm) circle (0.10cm);\fill (11cm,01cm) circle (0.10cm);\fill (12cm,01cm) circle (0.10cm);
	
	{\color{orange}
		\fill (2cm,12cm)circle(0.3cm);
	}
	{\color{red}
		\fill (4cm,10cm)circle(0.3cm);
	}
	{\color{blue}
		\fill (6cm,8cm)circle(0.3cm);
	}
	{\color{green}
		\fill (8cm,4cm)circle(0.3cm);
	}
	{\color{magenta}
		\fill (10cm,2cm)circle(0.3cm);
	}
	{\color{cyan}
		\fill (12cm,6cm)circle(0.3cm);
	}
\end{tikzpicture}
\caption{A covering of $\Z_6^2$ as in Lemma \ref{xi2mod4} and a covering of $\diag_{12}$ as in Lemma \ref{xid0mod4}}\label{fig-even-values}
\end{figure}

\begin{lemma}\label{xid0mod4}
If $n$ is a positive integer multiple of $4$, then $\xid(n)=n/2$.
\end{lemma}
\begin{proof} Let $4m:=n$. By Lemma \ref{lower bound}, it is enough to find a covering of $\diag_{4m}$ with $2m$ elements. Such covering is illustrated for $4m=12$ in Figure \ref{fig-even-values}. Define:
\begin{itemize}
\item $A=\{(2,4m),(4,4m-2),\dots (2m,2m+2)\}$,
\item $B:=\{(2m+2,2m-2),(2m+4,2m-4), \dots, (4m-2,2)\}$ and
\item $C:=\{(4m,2m)\}$
\end{itemize}
We claim that $\{\sq(x):x\in A\u B\u C\}$ covers $\diag_{4m}$. Let $(a,b)\in \mathbb D_{4m}$. If $a$ or $b$ is even, then it is clear that $(a,b)$ is in the horizontal or vertical line of a member of $A\u B\u C$. Suppose that both $a$ and $b$ are odd. We will use the function $\delta_k$ for $k=4m$. Now, $\delta_{4m}(a,b)$ is even. Moreover, $\delta_{4m}(a,b)\neq 0$, since $(a,b)\notin \mathbb D_{4m}$ if $a=b$. Note that $\delta_{4m}(A)=\{4m-2,4m-6,\dots,6,2\}$ and $\delta_{4m}(B)=\{-4,-8,\dots,4-4m\}$. So, $\delta_{4m}(A\u B)$ contains all non-zero even elements of $\Z_{4m}$. In particular it contains $\delta_{4m}(a,b)=b-a$. Therefore, $(a,b)\in D(x)\cont \sq(x)$ for some $x\in A\u B$ and the lemma holds.
\end{proof}

The next lemma is elementary and its proof is omitted.
\begin{lemma}\label{setplus1}
If $\emptyset\subsetneq S\subsetneq\zn$ then $S\neq\{x+1:x\in S\}$.
\end{lemma}

\begin{lemma}\label{consecutive}
Let $V_0$ and $V_1$ be consecutive vertical lines of $\zn^2$ with $n\ge 2$. Suppose that $X,Y \cont \zn^2$ satisfy $|Y|,|X|\le n-1$ and $V_0\u V_1\cont D(X)\u H(Y)$. Then $|X|+|Y|\ge n+1$.
\end{lemma}
\begin{proof}
Suppose the contrary. Say that $V_1=\{(a+1,b):(a,b)\in V_0\}$. For $i=0,1$, $V_i$ is the union of $A_i:=V_i\i D(X)$ and $B_i:=V_i\i H(Y)$. Note that $|A_0|=|A_1|\le|X|$ and $|B_0|=|B_1|\le|Y|$. For $i=0,1$, $|V_i|\le|A_i|+|B_i|=|X|+|Y|\le n=|V_i|$. So, $A_i\cap B_i=\emp$. Define a function $\pi:\zn^2\rightarrow \zn$ by $\pi(a,b)=b$. Since the restriction of $\pi$ to each vertical line is bijective, $\pi(A_0)=\zn-\pi(B_0)=\zn-\pi(B_1)=\pi(A_1)$. But $\pi(A_1):=\{t+1:t\in \pi(A_0)\}$, a contradiction to Lemma \ref{setplus1}.
\end{proof}

\begin{lemma}\label{lower xid odd}
	For each odd integer $n\ge 3$, $\xid(n)\ge (n+1)/2$.
\end{lemma}
\begin{proof}
	Let $2m+1:=n$ with $m\ge 1$. By Lemma \ref{lower bound}, $\xid(2m+1)\ge m$. Suppose for a contradiction that $\xid(2m+1)=m$ and let $X:=\{x_1,\dots,x_m\}$ be an $m$-subset of $\diag(2m+1)$ such that $\{\sq(x):x\in X\}$ covers $\diag(2m+1)$. Thus, there are two consecutive vertical lines $V_0$ and $V_1$ in $\Z_{2m+1}^2$ avoiding $V(X)$. As $\diag(2m+1)\cont \sq(X)$, it follows that $V_0\u V_1\cont H(X)\u D(X\u\{(1,1)\})$. By Lemma \ref{consecutive}, $2m+1=2|X|+1\ge (2m+1)+1$, a contradiction.
\end{proof}

The next lemma was proved by L. Euler~\cite{Euler}. An alternative proof may be found in \cite[Corollary 1]{Evans}. The reader also may see a more general result in Wanless's survey \cite[Theorem 2]{Wanless}, proved by Maillet~\cite{Maillet}.

\begin{lemma}(Euler, 1779)\label{Latin square}
Let $Q=[q_{ij}]$ be a Latin square with even order $n\ge 2$. Suppose that $q_{ij}=q_{kl}$ if and only if $i - j \equiv k - l \mod n$. Then, $Q$ admits no set $X$ of $n$ entries such that each pair of entries of $X$ are in different rows, different columns and has different symbols. (Such a set is called a \defin{Latin transversal}.)
\end{lemma}

In Lemma \ref{Latin square}, supposing that $q_{ij}=q_{kl}$ if and only if $i + j \equiv k + l \mod n$ has the same effect; usually this is the way it is stated.

\begin{lemma}\label{xi0mod4}
If $n$ is a positive integer multiple of four, then $\xi(n)=1+n/2$.
\end{lemma}
\begin{proof} Let $4m:=n$. By Lemma \ref{xid0mod4}, $2m=\xid(4m)\le \xi(4m) \le \xid(4m)+1=2m+1$. So, all we have to prove is that $\xi(4m)\neq 2m$. Suppose for a contradiction that $\xi(4m)= 2m$. Let $X:=\{(a_t,b_t):t=1,\dots,2m\}$ be a subset of $\Z_{4m}^2$ such that $\{\sq(x):x\in X\}$ covers $\Z_{4m}^2$.

First we will prove that:
\begin{equation}\label{eq-xi0mod4}
\{a_1,\dots,a_{2m}\},\{b_1,\dots,b_{2m}\}\in\big\{\{1,3,\dots,4m-1\},\{2,4,\dots,4m\}\big\}.
\end{equation}
Suppose the contrary. Then, there are two consecutive horizontal lines avoiding $H(X)$ or two consecutive vertical lines avoiding $V(X)$. We may assume the later case. Let $V_0$ and $V_1$ be such lines. So, $V_0\u V_1\cont H(X)\u D(X)$. By Lemma \ref{consecutive}, $2|X|=4m\ge4m+1$, a contradiction. So, \eqref{eq-xi0mod4} holds.

By \eqref{eq-xi0mod4}, we may assume, without loss of generality, that $E:=\{2,4,\dots,4m\}=\{a_1,\dots,a_{2m}\}=\{b_1,\dots,b_{2m}\}$. 
So $X\cont E\times E$. Let $F:=\Z_{4m}-E$. The fact that $V(X)\u H(X)$ does not intersect $F\times F$ implies that $F\times F\cont D(X)$. We will use the function $\delta_k$ \ref{delta function} for $k=4m$. Note that, in each row or column of $F\times F$, $\delta_{4m}$ assumes $2m$ distinct values. By Lemma \ref{delta function}, $\delta_{4m}$ also assumes $2m$ distinct values on $X$. Now, construct a Latin square having $F$ as set of rows and columns such that the symbol in $(a,b)$ is $\delta_{4m}(a,b)$. The existence of this Latin square in $X$ contradicts Lemma \ref{Latin square}.
\end{proof}

Next we prove Theorem \ref{prodvalues}.\\
	\begin{figure}[h]\centering
	\begin{tikzpicture}[scale=0.8]
	{\color{red} \node at(2cm,2cm) {$\blacksquare$};   
	\draw[thick] (2cm,0.5cm)--(2cm,5.5cm);
	\draw[thick] (0.5cm,2cm)--(5.5cm,2cm);
	\draw[thick] (0.5cm,0.5cm)--(5.5cm,5.5cm); }
	{\color{blue}\node at(3cm,4cm) {$\blacklozenge$}; 
	\draw[thick] (0.5cm,4cm)--(5.5cm,4cm);
	\draw[thick] (3cm,0.5cm)--(3cm,5.5cm);
	\draw[thick] (0.5cm,1.5cm)--(4.5cm,5.5cm);
	\draw[thick](4.5cm,0.5cm)--(5.5cm,1.5cm);}
	{\color{orange} \node at(4cm,3cm) {$\CIRCLE$}; 
	\draw[thick] (0.5cm,3cm)--(5.5cm,3cm);
	\draw[thick] (4cm,0.5cm)--(4cm,5.5cm);
	\draw[thick] (1.5cm,0.5cm)--(5.5cm,4.5cm);
	\draw[thick](0.5cm,4.5cm)--(1.5cm,5.5cm);}

	{\draw[thick](0.5cm,0.5cm)--(0.5cm,5.5cm);}
	{\color{gray}
	\draw(1.5cm,0.5cm)--(1.5cm,5.5cm);
	\draw(2.5cm,0.5cm)--(2.5cm,5.5cm);
	\draw(3.5cm,0.5cm)--(3.5cm,5.5cm);
	\draw(4.5cm,0.5cm)--(4.5cm,5.5cm);
	}
	{\draw [thick] (5.5cm,0.5cm)--(5.5cm,5.5cm);}
	{\draw [thick]         (0.5cm,0.5cm)--(5.5cm,0.5cm);}
	{\color{gray}
	\draw[] (0.5cm,1.5cm)--(5.5cm,1.5cm);
	\draw[] (0.5cm,2.5cm)--(5.5cm,2.5cm);
	\draw[] (0.5cm,3.5cm)--(5.5cm,3.5cm);
	\draw[] (0.5cm,4.5cm)--(5.5cm,4.5cm);        
	}
	{\draw [thick] (0.5cm,5.5cm)--(5.5cm,5.5cm);}
	
	\fill (01cm,05cm) circle (0.10cm);\fill (02cm,05cm) circle (0.10cm);\fill (03cm,05cm) circle (0.10cm);\fill (04cm,05cm) circle (0.10cm);\fill (05cm,05cm) circle (0.10cm);
	\fill (01cm,04cm) circle (0.10cm);\fill (02cm,04cm) circle (0.10cm);\fill (03cm,04cm) circle (0.10cm);\fill (04cm,04cm) circle (0.10cm);\fill (05cm,04cm) circle (0.10cm);
	\fill (01cm,03cm) circle (0.10cm);\fill (02cm,03cm) circle (0.10cm);\fill (03cm,03cm) circle (0.10cm);\fill (04cm,03cm) circle (0.10cm);\fill (05cm,03cm) circle (0.10cm);
	\fill (01cm,02cm) circle (0.10cm);\fill (02cm,02cm) circle (0.10cm);\fill (03cm,02cm) circle (0.10cm);\fill (04cm,02cm) circle (0.10cm);\fill (05cm,02cm) circle (0.10cm);
	\fill (01cm,01cm) circle (0.10cm);\fill (02cm,01cm) circle (0.10cm);\fill (03cm,01cm) circle (0.10cm);\fill (04cm,01cm) circle (0.10cm);\fill (05cm,01cm) circle (0.10cm);

	{\color{red} \node at(2cm,2cm) {$\blacksquare$};}
	{\color{blue}\node at(3cm,4cm) {$\blacklozenge$};}
	{\color{orange} \node at(4cm,3cm) {$\CIRCLE$};}
\end{tikzpicture}$\qquad$
\begin{tikzpicture}[scale=0.5]
\fill[gray!20,opacity=1] (0.5cm,0.5cm)--(0.5,5.5cm)--(5.5cm,5.5cm)--(5.5cm,0.5cm)--cycle;
\fill[gray!20,opacity=1] (5.5cm,10.5cm)--(10.5,10.5cm)--(10.5cm,15.5cm)--(5.5cm,15.5cm)--cycle;
\fill[gray!20,opacity=1] (10.5cm,5.5cm)--(15.5,5.5cm)--(15.5cm,10.5cm)--(10.5cm,10.5cm)--cycle;

{\color{red}
\draw[thick] (0.5cm,14cm)--(15.5cm,14cm);
\draw[thick] (0.5cm,10.5cm)--(5.5cm,15.5cm);
\draw[thick] (5.5cm,0.5cm)--(15.5cm,10.5cm);
\draw[thick] (9cm,0.5cm)--(9cm,15.5cm);
\draw[thick] (0.5cm,12cm)--(15.5cm,12cm);
\draw[thick] (0.5cm,11.5cm)--(4.5cm,15.5cm);
\draw[thick] (4.5cm,0.5cm)--(15.5cm,11.5cm);
\draw[thick] (8cm,0.5cm)--(8cm,15.5cm);
\draw[thick] (0.5cm,13cm)--(15.5cm,13cm);
\draw[thick] (0.5cm,9.5cm)--(6.5cm,15.5cm);
\draw[thick] (6.5cm,0.5cm)--(15.5cm,9.5cm);
\draw[thick] (7cm,0.5cm)--(7cm,15.5cm);
}
{\color{blue}
\draw[thick] (0.5cm,0.5cm)--(15.5cm,15.5cm);
\draw[thick] (0.5cm,1.5cm)--(14.5cm,15.5cm);
\draw[thick] (1.5cm,0.5cm)--(15.5cm,14.5cm);
\draw[thick] (4cm,0.5cm)--(4cm,15.5cm);
\draw[thick] (3cm,0.5cm)--(3cm,15.5cm);
\draw[thick] (2cm,0.5cm)--(2cm,15.5cm);
\draw[thick] (14.5cm,0.5cm)--(15.5cm,1.5cm);
\draw[thick] (0.5cm,14.5cm)--(1.5cm,15.5cm);
\draw[thick] (0.5cm,4cm)--(15.5cm,4cm);
\draw[thick] (0.5cm,3cm)--(15.5cm,3cm);
\draw[thick] (0.5cm,2cm)--(15.5cm,2cm); 
}
{\color{orange}
\draw[thick] (0.5cm,9cm)--(15.5cm,9cm);
\draw[thick] (0.5cm,7cm)--(15.5cm,7cm);
\draw[thick] (0.5cm,8cm)--(15.5cm,8cm);
\draw[thick] (10.5cm,0.5cm)--(15.5cm,5.5cm);
\draw[thick] (9.5cm,0.5cm)--(15.5cm,6.5cm); 
\draw[thick] (11.5cm,0.5cm)--(15.5cm,4.5cm); 
\draw[thick] (0.5cm,5.5cm)--(10.5cm,15.5cm);
\draw[thick] (0.5cm,6.5cm)--(9.5cm,15.5cm); 
\draw[thick] (0.5cm,4.5cm)--(11.5cm,15.5cm);
\draw[thick] (14cm,0.5cm)--(14cm,15.5cm);
\draw[thick] (13cm,0.5cm)--(13cm,15.5cm);
\draw[thick] (12cm,0.5cm)--(12cm,15.5cm); 
}                                                                                  
                                                                                   
\fill (01cm,15cm) circle (0.10cm);\fill (02cm,15cm) circle (0.10cm);\fill (03cm,15cm) circle (0.10cm);\fill (04cm,15cm) circle (0.10cm);\fill (05cm,15cm) circle (0.10cm);
\fill (01cm,14cm) circle (0.10cm);\fill (02cm,14cm) circle (0.10cm);\fill (03cm,14cm) circle (0.10cm);\fill (04cm,14cm) circle (0.10cm);\fill (05cm,14cm) circle (0.10cm);
\fill (01cm,13cm) circle (0.10cm);\fill (02cm,13cm) circle (0.10cm);\fill (03cm,13cm) circle (0.10cm);\fill (04cm,13cm) circle (0.10cm);\fill (05cm,13cm) circle (0.10cm);
\fill (01cm,12cm) circle (0.10cm);\fill (02cm,12cm) circle (0.10cm);\fill (03cm,12cm) circle (0.10cm);\fill (04cm,12cm) circle (0.10cm);\fill (05cm,12cm) circle (0.10cm);
\fill (01cm,11cm) circle (0.10cm);\fill (02cm,11cm) circle (0.10cm);\fill (03cm,11cm) circle (0.10cm);\fill (04cm,11cm) circle (0.10cm);\fill (05cm,11cm) circle (0.10cm);
\fill (01cm,10cm) circle (0.10cm);\fill (02cm,10cm) circle (0.10cm);\fill (03cm,10cm) circle (0.10cm);\fill (04cm,10cm) circle (0.10cm);\fill (05cm,10cm) circle (0.10cm);
\fill (01cm,09cm) circle (0.10cm);\fill (02cm,09cm) circle (0.10cm);\fill (03cm,09cm) circle (0.10cm);\fill (04cm,09cm) circle (0.10cm);\fill (05cm,09cm) circle (0.10cm);
\fill (01cm,08cm) circle (0.10cm);\fill (02cm,08cm) circle (0.10cm);\fill (03cm,08cm) circle (0.10cm);\fill (04cm,08cm) circle (0.10cm);\fill (05cm,08cm) circle (0.10cm);
\fill (01cm,07cm) circle (0.10cm);\fill (02cm,07cm) circle (0.10cm);\fill (03cm,07cm) circle (0.10cm);\fill (04cm,07cm) circle (0.10cm);\fill (05cm,07cm) circle (0.10cm);
\fill (01cm,06cm) circle (0.10cm);\fill (02cm,06cm) circle (0.10cm);\fill (03cm,06cm) circle (0.10cm);\fill (04cm,06cm) circle (0.10cm);\fill (05cm,06cm) circle (0.10cm);
\fill (01cm,05cm) circle (0.10cm);\fill (02cm,05cm) circle (0.10cm);\fill (03cm,05cm) circle (0.10cm);\fill (04cm,05cm) circle (0.10cm);\fill (05cm,05cm) circle (0.10cm);
\fill (01cm,04cm) circle (0.10cm);\fill (02cm,04cm) circle (0.10cm);\fill (03cm,04cm) circle (0.10cm);\fill (04cm,04cm) circle (0.10cm);\fill (05cm,04cm) circle (0.10cm);
\fill (01cm,03cm) circle (0.10cm);\fill (02cm,03cm) circle (0.10cm);\fill (03cm,03cm) circle (0.10cm);\fill (04cm,03cm) circle (0.10cm);\fill (05cm,03cm) circle (0.10cm);
\fill (01cm,02cm) circle (0.10cm);\fill (02cm,02cm) circle (0.10cm);\fill (03cm,02cm) circle (0.10cm);\fill (04cm,02cm) circle (0.10cm);\fill (05cm,02cm) circle (0.10cm);
\fill (01cm,01cm) circle (0.10cm);\fill (02cm,01cm) circle (0.10cm);\fill (03cm,01cm) circle (0.10cm);\fill (04cm,01cm) circle (0.10cm);\fill (05cm,01cm) circle (0.10cm);

\fill (06cm,15cm) circle (0.10cm);\fill (07cm,15cm) circle (0.10cm);\fill (08cm,15cm) circle (0.10cm);\fill (09cm,15cm) circle (0.10cm);\fill (10cm,15cm) circle (0.10cm);
\fill (06cm,14cm) circle (0.10cm);\fill (07cm,14cm) circle (0.10cm);\fill (08cm,14cm) circle (0.10cm);\fill (09cm,14cm) circle (0.10cm);\fill (10cm,14cm) circle (0.10cm);
\fill (06cm,13cm) circle (0.10cm);\fill (07cm,13cm) circle (0.10cm);\fill (08cm,13cm) circle (0.10cm);\fill (09cm,13cm) circle (0.10cm);\fill (10cm,13cm) circle (0.10cm);
\fill (06cm,12cm) circle (0.10cm);\fill (07cm,12cm) circle (0.10cm);\fill (08cm,12cm) circle (0.10cm);\fill (09cm,12cm) circle (0.10cm);\fill (10cm,12cm) circle (0.10cm);
\fill (06cm,11cm) circle (0.10cm);\fill (07cm,11cm) circle (0.10cm);\fill (08cm,11cm) circle (0.10cm);\fill (09cm,11cm) circle (0.10cm);\fill (10cm,11cm) circle (0.10cm);
\fill (06cm,10cm) circle (0.10cm);\fill (07cm,10cm) circle (0.10cm);\fill (08cm,10cm) circle (0.10cm);\fill (09cm,10cm) circle (0.10cm);\fill (10cm,10cm) circle (0.10cm);
\fill (06cm,09cm) circle (0.10cm);\fill (07cm,09cm) circle (0.10cm);\fill (08cm,09cm) circle (0.10cm);\fill (09cm,09cm) circle (0.10cm);\fill (10cm,09cm) circle (0.10cm);
\fill (06cm,08cm) circle (0.10cm);\fill (07cm,08cm) circle (0.10cm);\fill (08cm,08cm) circle (0.10cm);\fill (09cm,08cm) circle (0.10cm);\fill (10cm,08cm) circle (0.10cm);
\fill (06cm,07cm) circle (0.10cm);\fill (07cm,07cm) circle (0.10cm);\fill (08cm,07cm) circle (0.10cm);\fill (09cm,07cm) circle (0.10cm);\fill (10cm,07cm) circle (0.10cm);
\fill (06cm,06cm) circle (0.10cm);\fill (07cm,06cm) circle (0.10cm);\fill (08cm,06cm) circle (0.10cm);\fill (09cm,06cm) circle (0.10cm);\fill (10cm,06cm) circle (0.10cm);
\fill (06cm,05cm) circle (0.10cm);\fill (07cm,05cm) circle (0.10cm);\fill (08cm,05cm) circle (0.10cm);\fill (09cm,05cm) circle (0.10cm);\fill (10cm,05cm) circle (0.10cm);
\fill (06cm,04cm) circle (0.10cm);\fill (07cm,04cm) circle (0.10cm);\fill (08cm,04cm) circle (0.10cm);\fill (09cm,04cm) circle (0.10cm);\fill (10cm,04cm) circle (0.10cm);
\fill (06cm,03cm) circle (0.10cm);\fill (07cm,03cm) circle (0.10cm);\fill (08cm,03cm) circle (0.10cm);\fill (09cm,03cm) circle (0.10cm);\fill (10cm,03cm) circle (0.10cm);
\fill (06cm,02cm) circle (0.10cm);\fill (07cm,02cm) circle (0.10cm);\fill (08cm,02cm) circle (0.10cm);\fill (09cm,02cm) circle (0.10cm);\fill (10cm,02cm) circle (0.10cm);
\fill (06cm,01cm) circle (0.10cm);\fill (07cm,01cm) circle (0.10cm);\fill (08cm,01cm) circle (0.10cm);\fill (09cm,01cm) circle (0.10cm);\fill (10cm,01cm) circle (0.10cm);

\fill (11cm,15cm) circle (0.10cm);\fill (12cm,15cm) circle (0.10cm);\fill (13cm,15cm) circle (0.10cm);\fill (14cm,15cm) circle (0.10cm);\fill (15cm,15cm) circle (0.10cm);
\fill (11cm,14cm) circle (0.10cm);\fill (12cm,14cm) circle (0.10cm);\fill (13cm,14cm) circle (0.10cm);\fill (14cm,14cm) circle (0.10cm);\fill (15cm,14cm) circle (0.10cm);
\fill (11cm,13cm) circle (0.10cm);\fill (12cm,13cm) circle (0.10cm);\fill (13cm,13cm) circle (0.10cm);\fill (14cm,13cm) circle (0.10cm);\fill (15cm,13cm) circle (0.10cm);
\fill (11cm,12cm) circle (0.10cm);\fill (12cm,12cm) circle (0.10cm);\fill (13cm,12cm) circle (0.10cm);\fill (14cm,12cm) circle (0.10cm);\fill (15cm,12cm) circle (0.10cm);
\fill (11cm,11cm) circle (0.10cm);\fill (12cm,11cm) circle (0.10cm);\fill (13cm,11cm) circle (0.10cm);\fill (14cm,11cm) circle (0.10cm);\fill (15cm,11cm) circle (0.10cm);
\fill (11cm,10cm) circle (0.10cm);\fill (12cm,10cm) circle (0.10cm);\fill (13cm,10cm) circle (0.10cm);\fill (14cm,10cm) circle (0.10cm);\fill (15cm,10cm) circle (0.10cm);
\fill (11cm,09cm) circle (0.10cm);\fill (12cm,09cm) circle (0.10cm);\fill (13cm,09cm) circle (0.10cm);\fill (14cm,09cm) circle (0.10cm);\fill (15cm,09cm) circle (0.10cm);
\fill (11cm,08cm) circle (0.10cm);\fill (12cm,08cm) circle (0.10cm);\fill (13cm,08cm) circle (0.10cm);\fill (14cm,08cm) circle (0.10cm);\fill (15cm,08cm) circle (0.10cm);
\fill (11cm,07cm) circle (0.10cm);\fill (12cm,07cm) circle (0.10cm);\fill (13cm,07cm) circle (0.10cm);\fill (14cm,07cm) circle (0.10cm);\fill (15cm,07cm) circle (0.10cm);
\fill (11cm,06cm) circle (0.10cm);\fill (12cm,06cm) circle (0.10cm);\fill (13cm,06cm) circle (0.10cm);\fill (14cm,06cm) circle (0.10cm);\fill (15cm,06cm) circle (0.10cm);
\fill (11cm,05cm) circle (0.10cm);\fill (12cm,05cm) circle (0.10cm);\fill (13cm,05cm) circle (0.10cm);\fill (14cm,05cm) circle (0.10cm);\fill (15cm,05cm) circle (0.10cm);
\fill (11cm,04cm) circle (0.10cm);\fill (12cm,04cm) circle (0.10cm);\fill (13cm,04cm) circle (0.10cm);\fill (14cm,04cm) circle (0.10cm);\fill (15cm,04cm) circle (0.10cm);
\fill (11cm,03cm) circle (0.10cm);\fill (12cm,03cm) circle (0.10cm);\fill (13cm,03cm) circle (0.10cm);\fill (14cm,03cm) circle (0.10cm);\fill (15cm,03cm) circle (0.10cm);
\fill (11cm,02cm) circle (0.10cm);\fill (12cm,02cm) circle (0.10cm);\fill (13cm,02cm) circle (0.10cm);\fill (14cm,02cm) circle (0.10cm);\fill (15cm,02cm) circle (0.10cm);
\fill (11cm,01cm) circle (0.10cm);\fill (12cm,01cm) circle (0.10cm);\fill (13cm,01cm) circle (0.10cm);\fill (14cm,01cm) circle (0.10cm);\fill (15cm,01cm) circle (0.10cm);

{\draw[thick](0.5cm,0.5cm)--(0.5cm,15.5cm);}
{\color{gray}
	\draw(1.5cm,0.5cm)--(1.5cm,15.5cm);
	\draw(2.5cm,0.5cm)--(2.5cm,15.5cm);
	\draw(3.5cm,0.5cm)--(3.5cm,15.5cm);
	\draw(4.5cm,0.5cm)--(4.5cm,15.5cm);
	}
{\draw [thick] (5.5cm,0.5cm)--(5.5cm,15.5cm);}
{\color{gray}
	\draw[] (6.5cm,0.5cm)--(6.5cm,15.5cm);
	\draw[] (7.5cm,0.5cm)--(7.5cm,15.5cm);
	\draw[] (08.5cm,0.5cm)--(08.5cm,15.5cm);
	\draw[] (09.5cm,0.5cm)--(09.5cm,15.5cm);     
	}
{\draw [thick] (10.5cm,0.5cm)--(10.5cm,15.5cm);}
{\color{gray}
	\draw[] (11.5cm,0.5cm)--(11.5cm,15.5cm);
	\draw[] (12.5cm,0.5cm)--(12.5cm,15.5cm);
	\draw[] (13.5cm,0.5cm)--(13.5cm,15.5cm);
	\draw[] (14.5cm,0.5cm)--(14.5cm,15.5cm);     
	}
{\draw[thick]          (15.5cm,0.5cm)--(15.5cm,15.5cm);}

{\draw [thick]         (0.5cm,0.5cm)--(15.5cm,0.5cm);}
{\color{gray}
	\draw[] (0.5cm,1.5cm)--(15.5cm,1.5cm);
	\draw[] (0.5cm,2.5cm)--(15.5cm,2.5cm);
	\draw[] (0.5cm,3.5cm)--(15.5cm,3.5cm);
	\draw[] (0.5cm,4.5cm)--(15.5cm,4.5cm);        
	}
{\draw [thick] (0.5cm,5.5cm)--(15.5cm,5.5cm);}
{\color{gray}
	\draw[] (0.5cm,6.5cm)--(15.5cm,6.5cm);
	\draw[] (0.5cm,7.5cm)--(15.5cm,7.5cm);
	\draw[] (0.5cm,08.5cm)--(15.5cm,08.5cm);
	\draw[] (0.5cm,09.5cm)--(15.5cm,09.5cm);      
	}
{\draw [thick] (0.5cm,10.5cm)--(15.5cm,10.5cm);}
{\color{gray}
	\draw[] (0.5cm,11.5cm)--(15.5cm,11.5cm);
	\draw[] (0.5cm,12.5cm)--(15.5cm,12.5cm);
	\draw[] (0.5cm,13.5cm)--(15.5cm,13.5cm);
	\draw[] (0.5cm,14.5cm)--(15.5cm,14.5cm);      
	}
{\draw[thick] (0.5cm,15.5cm)--(15.5cm,15.5cm);}

{\color{red}
\node at(7cm,12cm) {$\blacksquare$};
\node at(8cm,14cm) {$\blacklozenge$};
\node at(9cm,13cm) {$\CIRCLE$};}
{\color{blue}
\node at(2cm,2cm) {$\blacksquare$};
\node at(3cm,4cm) {$\blacklozenge$};
\node at(4cm,3cm) {$\CIRCLE$};}
{\color{orange}
\node at(12cm,7cm) {$\blacksquare$};
\node at(13cm,9cm) {$\blacklozenge$};
\node at(14cm,8cm) {$\CIRCLE$};}
\end{tikzpicture}
\caption{A covering of  $\Z_{15}^2$ by semiqueens constructed from a covering of $\Z_5^2$ with the method of the proof of Theorem \ref{prodvalues}}\label{fig-xiproduct}
\end{figure}

\begin{proofof}\emph{Proof of Theorem \ref{prodvalues}: }
As we will deal with both rings $\Z_n$ and $\Z_{mn}$, for distinction purposes, we will denote by $z+k\Z$ the residue class modulo $k$ of $z$ for each $z\in \Z$ and $k\in\{n,mn\}$. Let $\{\sq_{\Z_n}(x_i+n\Z,y_i+n\Z):i=1,\dots,\xi(\Z_n)\}$ be a covering for $\Z_n^2$. Define 
$$X:=\big\{\big(
(x_i+\lambda n)+mn\Z,(y_i-\lambda n)+mn\Z\big):i=1,\dots,\xi(\Z_n) \text{ and }\lambda\in \Z\big\}.$$
Note that $|X|\le m\xi(\Z_n)$. To prove item (a), it suffices to check that each element $(a+mn\Z,b+mn\Z)\in\Z_{mn}^2$ is in the semiqueen of some element of $X$. Let $i$ be the index such that, for $v_i:=(x_i+n\Z,y_i+n\Z)$, $u:=(a+n\Z,b+n\Z)\in\sq_{Z_n}(v_i)$. 
If $u\in V_{Z_n}(v_i)$, then there is an integer $\lambda$ such that $a=x_i+\lambda n$ and $(a+mn\Z,b+mn\Z)$ is in the semiqueen of $\big((x_i+\lambda n)+mn\Z,(y_i-\lambda n)+mn\Z\big)\in X$. If $u\in H_{Z_n}(v_i)$, we proceed analogously. So, we may assume that $u\in D_{Z_n}(v_i)$. We will use the function $\delta_k$ for $k=n$ on $\Z_{n}^2$ and for $k=mn$ on $\Z_{mn}$. As $u\in D_{Z_n}(v_i)$, it follows that $\delta_n(u)=\delta_n(v_i)$ and, therefore, there is an integer $\alpha$ such that $b-a=y_i-x_i-\alpha n$.
Since $mn$ is odd, $2+mn\Z$ is invertible in $\Z_{mn}$ and there is an integer $\lambda$ such that $2\lambda+mn\Z=\alpha+mn\Z$. Now 
\begin{eqnarray*}
\delta_{mn}\big((x_i+\lambda n)+mn\Z,(y_i-\lambda n)+mn\Z\big)
&=&
(y_i-x_i)-2\lambda n+mn\Z\\
&=&
b-a+mn\Z\\
&=&
\delta_{mn}(a+mn\Z,b+mn\Z).
\end{eqnarray*}
By Lemma \ref{delta function}, $(a+mn\Z,b+mn\Z)$ is in the semiqueen of $\big((x_i+\lambda n)+mn\Z,(y_i-\lambda n)+mn\Z\big)\in X$. This completes the proof for item (a). Item (b) follows from item (a) and Theorem \ref{main}.
\end{proofof}

\begin{lemma}\label{oddxi}
If $n$ is odd, then $\xi(n)\le\lf(2n+1)/3\rf$.
\end{lemma}
\begin{proof}
Write $n=3m+r$ with $r\in\{0,1,2\}$. If $r=0$, then, as $\xi(3)=2$, by Theorem \ref{prodvalues}, $\xi(n)\le 2m= \lf(2n+1)/3\rf$ and the lemma holds. So, assume that $r\in\{1,2\}$. We shall prove that $\xi(n)\le 2m+1$. Consider a subdivision of the board as below:
$$\begin{array}{|c|c|c|}
	\hline
	Q_{1,3}&Q_{2,3}&Q_{3,3}\\
	\hline
	Q_{1,2}&Q_{2,2}&Q_{3,2}\\
	\hline
	Q_{1,1}&Q_{2,1}&Q_{3,1}\\
	\hline
  \end{array}\quad,
$$
where $Q_{1,1}$, $Q_{2,2}$ and $Q_{3,3}$ are square blocks with respective orders $m+1$, $m$ and $m+r-1$. Let $X_i$ be the set of the pairs of $\zn^2$ in the southeast-northwest diagonal of $Q_{i,i}$. Define $X=X_1\u X_2$. We will prove that $\{\sq(x):x\in X\}$ covers the board. If $Q_{i,j}\neq Q_{3,3}$ it is clear that $Q_{i,j}\cont {\bigcup_{x\in X}\big(H(x)\u V(x)\big)}$. So, let $y\in Q_{3,3}$. We shall prove that $y\in D(x)$ for some $x\in X$. If $c$ is a coordinate of $y$, then $2m+2\le c\le 3m+r$. Therefore, $-m\le-m-r+2\le\delta_n(y)\le m+r-2\le m$. Note that $\delta_n(X_1)=\{m,m-2,m-4,\dots,4-m,2-m,-m\}$ and $\delta_n(X_2)=\{m-1,m-3,\dots,3-m,1-m\}$. So, $\delta_n(X_1\u X_2)=\{m,m-1,\dots,1-m,-m\}$ and $\delta_n(y)\in \delta_n(X)$. By Lemma \ref{delta function}, $y$ is in $D(x)$ for some $x\in X$ and the lemma holds.
\end{proof}

\begin{proofof}\emph{Proof of Theorem \ref{xivalues}: }
	Items (a) and (b) follow from Lemmas \ref{xi2mod4} and \ref{xi0mod4}, respectively. Item (c) follows from Lemmas \ref{lower xid odd} and \ref{oddxi}.
\end{proofof}

\begin{proofof}\emph{Proof of Theorem \ref{xidvalues}: }
	Item (a) follows from Lemmas \ref{xi2mod4} and \ref{xid0mod4}. Item (b) follows from Lemmas \ref{lower xid odd} and \ref{oddxi}.
\end{proofof}

\section{From $\fq^3$ to the projective plane}\label{sec-projective}

In this section, we establish relations between short coverings and coverings by semiqueens using the projective plane as a link between them. We prove Theorem \ref{main} at the end of this section.

We define the \defin{projective plane} $\pq$ as the set of the $1$-dimensional vector subspaces of $\fq^3$; we call its elements \defin{points}. We say that $L\cont \pq$ is a \defin{line} if the union of the elements of $L$ is a $2$-dimensional vector subspace of $\fq^3$. We denote the subspace spanned by $(\alpha,\beta,\gamma)\in \fq^3-\{0\}$ by homogeneous coordinates $(\alpha:\beta:\gamma)\in \pq$.


We say that the points of $\pq$ are \defin{cardinal}, \defin{coast} or \defin{midland} when they have exactly one, two or three non-zero coordinates respectively. We denote the cardinal points by $c_1:=(1:0:0)$, $c_2:=(0:1:0)$ and $c_3:=(0:0:1)$. We also denote the line containing the points $u$ and $v$ by $\l{u,v}$, provided $u\neq v$, and, for convenience, $\l{u,u}:=\{u\}$. We say that a line of $\pq$ is a \defin{midland line} if it contains a midland point and a \defin{coast line} otherwise. Note that the unique coast lines are $\l{c_1,c_2}$, $\l{c_1,c_3}$ and $\l{c_2,c_3}$. Moreover, we denote by $e_i$ the $i$-th vector in the canonical basis of $\fq^3$, and by $[v_1,\dots,v_n]$ the subspace of $\fq^3$ spanned by $v_1,\dots,v_n$. We denote by $EB[v,r]$ the extended ball with radius $r$ and center $v$. The next lemma is easy to check:

\begin{lemma}\label{balls and windroses}
If $v\in \fq^3-\{0\}$, then $EB[v,1]$ is the union of the members of $\l{[v],c_1}\u \l{[v],c_2}\u \l{[v],c_3}$.
\end{lemma}

Motivated by Lemma \ref{balls and windroses}, we define the \defin{compass rose} of $p\in \pq$ as $$W(p):=\l{p,c_1}\u \l{p,c_2}\u \l{p,c_3}.$$
From Lemma \ref{balls and windroses}, we may conclude:

\begin{corollary}\label{equivalent covers 1}
Let $p_1,\dots,p_n\in \pq$ and for $i=1,\dots, n$, let $v_i\in p_i-\{0\}$. Then, $\{W(p_1),\dots,W(p_n)\}$ covers $\pq$ if and only if $\{EB[v_1,1],\dots, EB[v_n,1]\}$ covers $\fq^3$. Moreover, $c(q)$ is the size of a minimum covering of $PG(2,q)$ by compass roses.
\end{corollary}

We say that a compass rose $W(p)$ is \defin{cardinal}, \defin{coast} or \defin{midland} according to which of these adjectives applies to $p$. It is clear that $W(p_1)=W(p_2)$ implies $p_1=p_2$. So, exactly one of these adjectives applies to a particular compass rose. The following properties of compass roses are elementary and easy to check:

\begin{lemma}\label{compass roses characterizations} Each midland compass rose is the union of three distinct midland lines, each coast compass rose is the union of a coast and a midland line and each cardinal compass rose is the union of two distinct coast lines.
\end{lemma}

\begin{lemma}\label{covering coast}
Let $q$ be a prime power and suppose that $c(q)\le q-2$. Then, every minimum covering $\C$ of $\pq$ by compass roses contains at least two non-midland compass roses. In particular, each coast line is contained in a member of $\C$.
\end{lemma}
\begin{proof} Since $\pq$ has three distinct coast lines, the first part of the lemma follows from the second part and from Lemma \ref{compass roses characterizations}. So, let us prove the second part. Suppose that it fails. Let $\C$ be a minimum covering of $\pq$ by compass roses  such that no member contains a fixed coast line $L$. Let $K$ be the set of coast points in $L$. Since no member of $\C$ contains $L$, each compass rose in $\C$ meets $K$ in at most one point. So, $q-1=|K|\le |\C|=c(q)\le q-2$, a contradiction.
\end{proof}

\begin{lemma}\label{particular cover} Let $q$ be a prime power and suppose that $c(q)\le q-2$. Then, there is a minimum covering of $\pq$ by compass roses containing precisely one cardinal compass rose and one coast compass rose.
\end{lemma}
\begin{proof}
Choose a minimum covering $\C$ of $\pq$ by compass roses maximizing the number of midland compass roses primarily and coast compass roses secondarily. There are three coast lines in $\pq$: the members of $\L:=\{\l{c_1,c_2},\l{c_1,c_3},\l{c_2,c_3}\}$. By Lemmas \ref{covering coast} and \ref{compass roses characterizations}, the members of $\L$ are covered by:
\begin{enumerate}
 \item [(i)] One cardinal and one coast compass rose of $\C$,
 \item [(ii)] Two cardinal compass roses of $\C$, or
 \item [(iii)] Three coast compass roses of $\C$.
\end{enumerate}
We shall prove that (i) occurs. Indeed, first suppose for a contradiction that (ii) holds. Say that the members of $\L$ are covered by $W(c_1)$ and $W(c_2)$. If $p$ is a coast point of $\l{c_2,c_3}$, then $W(c_1)$ and $W(p)$ are enough to cover the coast lines of $\pq$. Hence $(\C-W(c_2))\u W(p)$ contradicts the secondary maximality of $\C$. Thus, (ii) does not hold.

Now, suppose that (iii) holds. The coast lines of $\pq$ are covered by three coast compass roses \linebreak$W(p_1),W(p_2),W(p_3)\in \C$. It is clear that $p_1$, $p_2$ and $p_3$ are in different coast lines. For $\{i,j,k\}=\{1,2,3\}$, say that $p_k\in \l{c_i,c_j}$. Let $x$ be the intersection point of $\l{c_2,p_2}$ and $\l{c_3,p_3}$. Note that $x$ is a midland point.  We claim that 
$$\C':=(\C-\{W(p_2),W(p_3)\})\u \{W(c_1),W(x)\}$$
contradicts the primary maximality of $\C$. Note that $\C'$ has more midland compass roses than $\C$ and $|\C'|\le |\C|$. It is left to to show that $\C'$ covers $\pq$. For this purpose, it is enough to prove 
that $W(p_2)\u W(p_3)\cont W(c_1)\u W(x)$. Indeed, as $p_2\in \l{c_1,c_3}$, it follows that $W(p_2)=\l{c_1,c_3}\u \l{c_2,p_2}$, but $\l{c_2,p_2}=\l{c_2,x}\cont W(x)$ and $\l{c_1,c_3}\cont W(c_1)$. Moreover, $W(p_3)=\l{c_1,c_2}\u \l{c_3,p_3}$, but $\l{c_3,p_3}=\l{c_3,x}\cont W(x)$  and $\l{c_1,c_2}\cont W(c_1)$. So, $\C'$ covers $PG(2,q)$ and (iii) does not occur. Therefore, (i) holds.

Now, let $W_1$ and $W_2$ be the respective compass roses described in (i). It is left to prove that $W$ is midland if $W\in \C -\{W_1, W_2\}$. As all cardinal compass roses are contained in $W_1\u W_2$, by the minimality of $\C$, it follows that $W$ is not cardinal. If $W$ is coast, then $W$ is the union of a coast line $C$ and a midland line $M$. But, $C\cont W_1\u W_2$ and, if $x\in M-C$ is a midland point, then $M\cont W(x)$. Thus $(\C-W)\u W(x)$ violates the primary maximality of $\C$. Therefore, $W$ is midland and the lemma holds.
\end{proof}

We define a bijection $f:\pq\rightarrow\pq$ to be a \defin{projective automorphism} if $f\big(\l{x,y}\big)=\l{f(x),f(y)}$ for all $x,y\in \pq$.

\begin{lemma}\label{isomorphism-isometry}
If $f$ is a projective automorphism of $\pq$ carrying cardinal points into cardinal points, then $f(W(x))=W(f(x))$ for each $x\in \pq$. Moreover, $x$ is midland (resp. coast, cardinal) if and only if $f(x)$ is midland (resp. coast, cardinal).
\end{lemma}
\begin{proof}
For $x\in\pq$:
\begin{eqnarray*}
f(W(x))&=&f\left(\l{x,c_1}\u\l{x,c_2}\u\l{x,c_3}\right)\\
&=&f\left(\l{x,c_1}\right)\u f\left(\l{x,c_2}\right)\u f\left(\l{x,c_3}\right)\\ 
&=&\l{f(x),f(c_1)}\u \l{f(x),f(c_2)}\u \l{f(x),f(c_3)}\\
&=&\l{f(x),c_1}\u \l{f(x),c_2}\u \l{f(x),c_3}\\ 
&=&W(f(x)).\end{eqnarray*}
This proves the first part of the lemma. For the second part, by hypothesis, $x$ is cardinal if and only if $f(x)$ is cardinal. Also, $x$ is coast if and only if $x$ is not cardinal but is in the line containing two cardinal points, thus $f(x)$ is also coast. Therefore, $x$ is coast if and only $f(x)$ is coast. By elimination, this implies that $x$ is midland if and only if $f(x)$ is midland.
\end{proof}

The next lemma has a straightforward proof.
\begin{lemma}\label{oposed}
Let $x$ be a cardinal point and $y$ a coast point of $PG(2,q)$ such that $W(x)\u W(y)$ contains all coast points of $PG(2,q)$. Then, for some $\{i,j,k\}=\{1,2,3\}$, $x=c_i$ and $y\in \l{c_j,c_k}$.
\end{lemma}

\begin{lemma}\label{no loss of generality}
Let $q$ be a prime power and suppose that $c(q)\le q-2$. There is a minimum covering of $\pq$ by compass roses containing $W(0:0:1)$ and $W(1:1:0)$ and such that all other members are midland.
\end{lemma}
\begin{proof}
By Lemma \ref{particular cover}, there is a minimum covering $\C$ of $\pq$ by compass roses, all of which are midland, except for two, namely $W(x)$ and $W(y)$, where $x$ is a cardinal point and $y$ a coast point. We may define a projective automorphism $f:\pq\rightarrow \pq$ by permutations of homogeneous coordinates and multiplying fixed coordinates by non-zero factors such that $f(x)=(0:0:1)$. By Lemma \ref{oposed}, $f(y)$ is in the form $(a:b:0)$ with $a\neq 0 \neq b$. So, in addition, we may pick $f$ in such a way that $f(y)=(1:1:0)$. By Lemma \ref{isomorphism-isometry}, $\{f(W):W\in \C\}$ is the covering we are looking for.
\end{proof}

Consider the multiplicative group $\fq^*$. We will use the terminologies $\mathbb D(\fq^*), \xi_D(\fq^*)$, etc. as defined in the beginning of Section \ref{sec-proofs} for $G=\fq^*$.

Consider the set $M$ of midland points in $PG(2,q)$ and the bijection $\psi$ between $(\fq^*)^2$ and $M$ defined by $\psi(a,b)=(a:b:1)$. For $x=(a,b)\in (\fq^*)^2$, we clearly have $\psi\big(H_{\fq^*}(x)\big)=M\i\big(\l{\psi(x),c_1}\big)$ and $\psi\big(V_{\fq^*}(x)\big)=M\i\big(\l{\psi(x),c_2}\big)$. Moreover, as $D_{\fq^*}(x)=\{(ta,tb):t\in \fq^*\}$, hence:
$$\psi\big(D_{\fq^*}(x)\big)=\{(ta:tb:1):t\in\fq^*\}=\{(a:b:t^{-1}):t\in\fq^*\}=M\i \big(\l{\psi(x),c_3}\big).$$
As a consequence, $\psi\big(\sq_{\fq^*}(x)\big)=M\i W(\psi(x))$. Note that $\psi\big(D_{\fq^*}(1,1)\big)=M\i \big(\l{c_3,(1:1:0)}\big)$. Therefore, the following lemma holds:

\begin{lemma} \label{wind semiqueens link} Consider the function $\psi$ as defined above and let $X\cont (\fq^*)^2$. Then 
$\sq_{\fq^*}(X)$ is a covering by semiqueens of $\diag(\fq^*)$ if and only if $\{W(\psi(x)):x\in X\}\u\{W(0:0:1),W(1:1:0)\}$ is a covering of $\pq$ by compass roses.
\end{lemma}

Now we are ready to prove Theorem \ref{main}.\\

\begin{proofof}\emph{Proof of Theorem \ref{main}: } It is well known that the multiplicative group of a finite field is cyclic. Thus, $\fq^*\cong \Z_{q-1}$ and, by Lemma \ref{isomorphism}, $\xid(\fq^*)=\xid(q-1)$. 

Let $\mathcal{Q}$ be a minimum covering of $\mathbb D(\fq^*)$ by semiqueens of $(\fq^*)^2$ and $\mathcal{R}$ the covering of $\pq$ by compass roses obtained from $\mathcal{Q}$ as in Lemma \ref{wind semiqueens link}. So, $|\mathcal{R}|-2=|\mathcal{Q}|=\xid(\fq^*)=\xid(q-1)$. By Corollary \ref{equivalent covers 1}, $c(q)\le |\mathcal R|=\xid(q-1)+2$.

Now we have to prove that $\xid(q-1)\le c(q)-2$ to finish the proof. When $q=5$ the values are known and match the theorem (see Section \ref{sec-ilp}). Assume that $q\ge 7$.

We shall prove next that $c(q)\le q-2$ in order to satisfy the hypothesis of Lemma \ref{no loss of generality}. First suppose that $q$ is odd. By Theorem \ref{xidvalues}, $\xid(q-1)\le (q-1)/2$. By the inequality that we already proved, $c(q)\le \xid(q-1)+2\le (q-1)/2 +2$. Since $q\ge 7$, this implies $c(q)\le q-2$. Now suppose that $q$ is even. It is known that $c(8)=6$ (see Section \ref{sec-ilp}). So, we may assume that $q\ge 16$. By Theorem \ref{xidvalues}, $\xid(q-1)\le (2q-1)/3$. Hence $c(q)\le \xid(q-1)+2\le (2q+5)/3$. This implies that $c(q)\le q-2$ because $q\ge 16$. Therefore, $c(q)\le q-2$ for each prime power $q\ge 7$.

By Corollary \ref{equivalent covers 1}, $c(q)$ is the size of a minimum covering $\{W(p):p\in A\}$ of $\pq$ by compass roses. By Lemma \ref{no loss of generality}, we may choose $A$ in such a way that $(0:0:1)$ and $(1:1:0)$ are in $A$ and all points of $B:=A-\{(0:0:1),(1:1:0)\}$ are midland. Consider the injective function $\psi:(\fq^*)^2\rightarrow\pq$ defined by $\psi(a,b)=(a:b:1)$, the same one of Lemma \ref{wind semiqueens link}. By Lemma \ref{wind semiqueens link}, $\mathcal Q:=\{W(x):x\in \psi^{-1}(B)\}$ is a covering of $\mathbb D(\fq^*)$ by semiqueens of $(\fq^*)^2$. So, $\xi_D(q-1)=\xi_D(\fq^*)\le |\mathcal Q|=|B|=|A|-2=c(q)-2$.
\end{proofof}


\section{Particular instances and ILP formulation}\label{sec-ilp}

For $X\in \{\zn^2,\diag_n\}$, the following integer 0-1 linear program may be used to find minimum coverings of $X$ by semiqueens of $\zn^2$. In this formulation, $x_p=1$  if and only if $\sq(p)$ is used in the covering.
\begin{equation*}\begin{array}{rl}
	{\rm Minimize: } 	& \sum\limits_{p\in \zn^2} x_p\\
	{\rm Subject\,\, to:}	& \forall q\in X:\, \sum \limits_{p\in \zn^2:q\in \sq(p)} x_p\ge 1.	
\end{array}\end{equation*}
For finding short coverings of $\fq^3$, a formulation in terms of compass roses in $PG(2,q)$ works similarly (see Corollary \ref{equivalent covers 1}):
\begin{equation*}\begin{array}{rl}
	{\rm Minimize: } 	& \sum\limits_{p\in PG(2,q)} x_p\\
	{\rm Subject\,\, to:}	& \forall q\in PG(2,q):\, \sum \limits_{p\in PG(2,q):q\in W(p)} x_p\ge 1.	
\end{array}\end{equation*}	

Some instances not covered by our theorems were solved using GLPK \cite{GLPK}, Cplex \cite{Cplex} and Gurobi \cite{Gurobi}. They are displayed in the tables below. The values $c(2)$, $c(3)$ and $c(4)$ are already known from \cite{Nakaoka}. \\

\begin{center}
$\begin{array}{|c|c|c|c|c|c|c|c|c|c|}\hline
n	&3 &5 &7 &9 &11	&13 &15 &17 \\\hline
\xi(n)	&2 &3 &5 &6 &7	&9  &9  &11 \\\hline
\xid(n)	&2 &3 &4 &6 &7	&8  &9  &11 \\\hline
\end{array}
$ $\qquad
\begin{array}{|c|c|c|c|c|c|c|c|c|c|}\hline
q		&2 &3 &4 &5 &8 &16 \\\hline
c(q)		&1 &3 &3 &4 &6 &11 \\\hline
\end{array}$
\end{center}


\begin{thebibliography}{30}

\bibitem{BMC} A.P. Burger, C.M. Mynhardt and E.J. Cockayne, Domination and Irredundance in the Queen's Graph, {\it Discrete Math.} {\bf 163} (1997), 47--66.

\bibitem{BMC-94} A.P. Burger, C.M. Mynhardt and E.J. Cockayne, Domination Numbers for the Queen's Graph, {\it Bulletin of the ICA} {\bf 10} (1994), 73--82.

\bibitem{Mynhardt}C.M. Mynhardt, Upper bounds for the domination numbers of toroidal queens graphs, {\it Discuss. Math. Graph Theory} {\bf 23} (2003), 163--175.

\bibitem{Burger2003} A.P. Burger, C.M. Mynhardt and W.D. Weakley, The domination number of the toroidal queens graph of size $3k \times 3k$, {\it Australas. J. Combin.} {\bf 28} (2003), 137--148.

\bibitem{Burger2001} A.P. Burger, E. J. Cockayne and C. M. Mynhardt, Queens graphs for chess-boards on the torus, {\it Australas. J. Combin.} {\bf 24} (2001), 231--246.

\bibitem{Cockayne} E.J. Cockayne, Chessboard domination problems, {\it Discrete Math.} {\bf 86} (1990), 13--20.

\bibitem{Cohen} G. Cohen, I. Honkala, S. Litsyn, and A. Lobstein, \emph{Covering Codes}, North-Holland, Amsterdam, 1997.

\bibitem{Euler} L. Euler, Recherche sur une nouvelle esp\`ece de quarr\`es magiques, {\it Leonardi Euleri Opera Omnia} {\bf 7} (1923) 291--392.

\bibitem{Evans} A.B. Evans, Latin squares without orthogonal mates, {\it Des. Codes Cryptogr.} {\bf 40} (2006), 121--130.

\bibitem{GLPK} GNU Linear Programming Kit, Version 4.52,
http://www.gnu.org/software/glpk/glpk.html 

\bibitem{Gurobi} Gurobi Optimization, Inc., Gurobi Optimizer Reference Manual, 2015, http://www.gurobi.com

\bibitem{Cplex} IBM ILOG CPLEX Optimizer, 2010, {http://www-01.ibm.com/software/integration/optimization/cplex-optimizer/} 

\bibitem{Keri} G. K\'eri, {\it Tables for bound on covering codes}, (2017, accessed) http://archive.is/8LCx

\bibitem{Maillet} E. Maillet, Sur les carr\'es Latins d'Euler, {\it  Assoc. Franc. Caen.} {\bf 23} (1894), 244--252.

\bibitem{Mart-tese} A.N. Martinh\~ao, \emph{Problemas combinat\'orios envolvendo estruturas em espa\c cos de Hamming}, Universidade Estadual de Maring\'a, PhD. Thesis, 2016.

\bibitem{Martinhao} A.N. Martinh\~ao and E.L. Monte Carmelo, Short covering codes arising from matchings in weighted graphs, {\it Math. of Comp.} {\bf 82} (2012), 605--616.

\bibitem{Mendes} C. Mendes, E.L. Monte Carmelo and M.V. Poggi, Bounds for short covering codes and reactive tabu search, {\it Discrete Appl. Math.} {\bf 158} (2010), 522--533.

\bibitem {Neto} E.L. Monte Carmelo and C.F.X. De Mendon\c ca Neto, Extremal problems on sum-free sets and coverings in tridimensional spaces, {\it Aequationes Math.} {\bf 78} (2009), 101--112.

\bibitem{Nakaoka}E.L. Monte Carmelo and I.N. Nakaoka, Short coverings in tridimensional spaces arising from sum-free sets, {\it European J. Combin.} {\bf 29} (2008), 227--233.

\bibitem{Geronimo} E.L. Monte Carmelo, I.N. Nakaoka and J.R. Geronimo, A covering problem on finite spaces and rook domains, {\it Int. J. Appl. Math.} {\bf 20} (2007), 875--886.

\bibitem{Wanless} I.M. Wanless, ``Transversals in Latin squares: A survey'', in .~Chapman (ed.), {\it Surveys in Combinatorics 2011}, London Math. Soc. Lecture Note Series 392, Cambridge University Press, 2011, 403--437.
\end{thebibliography}
\end{document}